\documentclass[11pt,reqno]{amsart}
\usepackage{amsbsy,amsfonts,amsmath,amssymb,amscd,amsthm,mathrsfs}
\usepackage{graphicx,epsfig,color}
\usepackage[a4paper,text={6.1in,8in},centering,includefoot,foot=1in]{geometry}

\newtheorem{theorem}{Theorem}[section]
\newtheorem*{theorem A}{Theorem A}
\newtheorem*{theorem B}{Theorem B}
\newtheorem*{theorem C}{Theorem C}

\newtheorem{lemma}[theorem]{Lemma}

\newtheorem{definition}[theorem]{Definition}
\newtheorem{corollary}[theorem]{Corollary}

\begin{document}
\title[Some stochastic properties of topological action semigroups]{Some stochastic Properties of topological action semigroups}
\author[Ghane]{F. H. Ghane}
\address{\centerline{Department of Mathematics, Ferdowsi University of Mashhad, }
\centerline{Mashhad, Iran.}  }
\email{ghane@math.um.ac.ir}
\author[Sarizadeh]{A. Sarizadeh}
\address{\centerline{Department of Mathematics, Ilam University, }
\centerline{Ilam, Iran.}  }
\email{ali.sarizadeh@gmail.com}
\begin{abstract}
For an iterated function system, we investigate combinatorial issues relating to the use of random
orbit and also we study stochastic properties of random orbits in a topological semigroup action.
In this area, the following question arises:
\begin{quote}
\emph{Under which conditions, the probability of branches satisfying a property
$\mathcal{Q}$ belongs to $\{0,1\}$, when $\mathcal{Q}$ is a property on fiber-wise orbits?}
\end{quote}
In this work, we will address that how the chaos game orbits satisfy the property $\mathcal{Q}$ almost surely.

As a consequence, without any regularity, we will provide some conditions for the action of semigroups on the sphere which ensure that they are proximal in a strong way. This means that
the limit infimum of distance between random images of each two points tends to zero
almost surely.
\end{abstract}

\maketitle
\thispagestyle{empty}

\section{Introduction}
An iterated function system can also be thought of as a finite collection of maps which can
be applied successively in any order and it
is a popular way
to generate and explore a variety of fractals.
To be more precise, we assume that $(X,d)$ is a Hausdorff topological space and
$\mathcal{F}=\{f_{1},\dots,f_{k}\}$ is a finite family of continuous maps defined on $X$. We denote by
$<\mathcal{F}>^+$ the semigroup generated by these maps.
Then the action of the semigroup $<\mathcal{F}>^+$ on $X$ is called the \emph{iterated function system}
associated to $\mathcal{F}$ and we denote it by $\textnormal{IFS}(X;\mathcal{F})$ or $\textnormal{IFS}(\mathcal{F})$. Throughout this paper we assume that $X$ is a compact metric space.

Let us consider the symbol spaces $\Sigma^+_k = \{1,\dots, k \}^\mathbb{N}$ and $\Sigma_k = \{1,\dots, k \}^\mathbb{Z}$.
suppose that the maps $f_i$ are randomly picked with probability $p_i > 0$ and $\mathbb{P}^+$ and $\mathbb{P}$ are the product measures on
$\Sigma^+_k $ and $\Sigma_k$, respectively, for these given probabilities
on $\{1,\dots,k \}$. Then this yields to a \emph{random
iterated function system} which is denoted by
$$
    \textnormal{RIFS}(\mathcal{F},\mathcal{P})=(X;f_1,\dots,f_k;p_1,\dots,p_k),
$$
where $\mathcal{P}=\{p_1,\dots,p_k\}$ with $\sum_{i=1}^k p_i=1$.\\
For any sequence
$\omega=(\omega_1\omega_2\dots\omega_n\dots)\in\Sigma^+_k$, we take
$$f^n_{\omega}(x):= f_{\omega_n}\circ f_\omega^{n-1}(x), \textnormal{for} \ \textnormal{all} \ n\in \mathbb{N}, \ \textnormal{and} \ f^{0}_{\omega}:=Id.$$
  Write
$f^{n}_{\omega}(x)=f^{n}_{\omega_{1}\dots\omega_{n}}(x).$
A \emph{fibrewise orbit} or \emph{chaos game orbit} corresponding to a one-way word
$\omega=(\omega_1\omega_2\dots\omega_n\ldots)\in\Sigma^+_k$ at point
$x$ is defined  by
$\mathcal{O}^+(x,\omega)=\{f^n_\omega(x)\}_{n=1}^\infty$.
The chaos game maybe used in data analysis \cite{j} and computer graphics \cite{n}.

Furthermore, we say that $\textnormal{IFS}(X;\mathcal{F})$ is
\emph{minimal} if any orbit has a branch which is dense in $X$, where the \emph{orbit} at a point $x$ is defined by
$$
\mathcal{O}^+(x)=\{h(x):\ h\in \textnormal{IFS}(X;\mathcal{F})\}=\bigcup_{\omega\in\Sigma^+_k}\mathcal{O}^+(x,\omega).
$$

Recently in \cite{bv} and \cite{br}, the authors studied the limit sets of IFSs and provided some conditions which ensure the minimality.

In fact in \cite{bv}, the authors obtained almost sure density
of random orbits in very general situations whenever the IFS is minimal.

 In the present paper, we essentially focus
on a generalization  of the same subject in whereas the density of fiber-wise orbits is a special case
(for more details see Section \ref{aaaa}).

To organize  main results, we need to introduce some notations and several definitions. Also we recall
elementary observations concerning to topological random iterated function systems.

Hereafter, let $\mathcal{E}, \tilde{\mathcal{E}}$ be two relations on $X$.
Denote by $a\mathcal{E}b$ (resp. $a\tilde{\mathcal{E}}b$) for two elements $a$ and $b$ of $X$ with  $(a,b)\in \mathcal{E}$ (resp. $(a,b)\in\tilde{\mathcal{E}}$). With respect to the relation $\mathcal{E}$, we define a general form of a chain.
Given $\omega\in\Sigma^+_k$, a sequence $\{x_i\}_{i=0}^n$
with $x_0=a,\ x_n=b$ is called a chain in direction of $\omega$
with respect to the relation $\mathcal{E}$ from $a$ to $b$ if
$$
       f_{\sigma^{i}(\omega)}(x_i )\mathcal{E}x_{i+1};\ \ \forall\ i=0,1,\dots,n-1,
$$
where $\sigma:\Sigma^+_k\to \Sigma^+_k$ is the Bernoulli shift map.
In this case, we use the notation $a\mathcal{E}^{n}_\omega b$ (or $a\mathcal{E}^{n}_{\omega_1\dots\omega_n}b$).
\begin{definition}
Let $A$ and $\mathcal{Q}$ be two subsets of $X$ and let  $\mathcal{E}, \tilde{\mathcal{E}}$ be two relations on $X$.
The set  $A$ has chain connection to $\mathcal{Q}$ with respect to the relations
$(\mathcal{E},\tilde{\mathcal{E}})$ if there exist points
$a\in A,~q\in \mathcal{Q}$, $\omega\in \Sigma^+_k$ and  a sequence of points $a=x_0,x_1,\dots, x_{n+1}=q$ so that
 $\{x_i\}_{i=0}^n$ is a chain in direction of $\omega$
with respect to the relation $\mathcal{E}$ from $x_0=a$ to $x_n$ and
 $\{x_{n+j}\}_{j=0}^1$ is a chain in direction of $\sigma^n(\omega)$
with respect to the relation $\tilde{\mathcal{E}}$ from $x_n$ to $x_{n+1}=b$.
 \end{definition}
We denote this composition of chains by  $a\mathcal{E}^{n}_\omega\tilde{\mathcal{E}}_{\sigma^n(\omega)}q$ (or $a\mathcal{E}^{n}_{\omega_1\dots\omega_n}\tilde{\mathcal{E}}_{\omega_{n+1}}q$).
\begin{definition}
Let  $X$ be a locally connected compact metric space and  $A$ and $\mathcal{Q}$ be two subsets of $X$. Also let $a\mathcal{E}^{n}_\omega\tilde{\mathcal{E}}_{\sigma^n(\omega)}q$
for some points $a\in A,~q\in \mathcal{Q}$.
If there exists an open set $U$ of the point $a$ so that for all points $u\in U$,
$u\mathcal{E}^{n}_\omega\tilde{\mathcal{E}}_{\sigma^n(\omega)}q$ then we say that $a$ has a stable chain connection to $\mathcal{Q}$ with respect to the relations $(\mathcal{E},\tilde{\mathcal{E}})$.
\end{definition}
\begin{definition}
A point $x\in X$ has a syndetic chain connection to $\mathcal{Q}$, for $\textnormal{IFS}(X;f_1,\ldots,f_k)$, if  the set
$$
\{i\in\mathbb{N}:\ x\mathcal{E}^{i}_\omega\tilde{\mathcal{E}}q,\ \textnormal{for} \ \textnormal{some} \ q\in \mathcal{Q} \ \textnormal{and} \ \omega\in\Sigma^+_k
 \}
$$
is syndetic in $\mathbb{N}$.
\end{definition}
Now we state the main results of this paper.
\begin{theorem A}\label{thmA} Let $ \mathcal{F}=\{f_1,\dots,f_k\}\subset\mathrm{Hom}(X)$
be a finite family of homeomorphisms on a locally connected  compact metric space $X$ and $\mathcal{Q}$ be a subset of $X$.
If  every point $x$ has a  stable chain connection to $\mathcal{Q}$ with respect to relations $(\mathcal{E},\tilde{\mathcal{E}})$ in some direction then for every $x\in X$  there exists
$\Omega(x)\subset\Sigma^+_k$ with $\mathbb{P}^+(\Omega(x)) = 1$ such that $x$ has a chain connection to $\mathcal{Q}$ in direction of $\omega$, for every $\omega\in\Omega(x)$.
\end{theorem A}
An extension of the above result can be obtained whenever the shift map acts on two sided time $\mathbb{Z}$.
To be more precise, for any sequence $\omega=(\dots\omega_{-1}\omega_0\omega_1\omega_2\dots)\in\Sigma_k$ and for $n\in\mathbb{N}$, we take
$$
f^n_{\omega}(x):= f_{\omega_{n-1}}\circ\dots\circ f_{\omega_{0}}(x)\ \ \text{and}\ \ f^{-n}_\omega(x):=(f^{-1})^n_{\omega_{-1}\omega_{-2}\dots}(x)= f^{-1}_{\omega_{-n}}\circ\dots\circ f^{-1}_{\omega_{-1}}(x) ,
$$
where $f^{0}_{\omega}:=Id.$
In this case, a \emph{fibre-wise orbit} corresponding to a two-way infinite word
$\omega\in\Sigma_k$ at point
$x$ is defined  by
$\mathcal{O}(x,\omega)=\{f^n_\omega(x)\}_{n=-\infty}^\infty$.
Denote shortly notation $a\mathcal{E}^{-n}_{\beta_{-n},\dots,\beta_{-1}}\tilde{\mathcal{E}}^{-1}b$ for
the sequence $\{x_i\}_{i=0}^n$ with $x_0=a,\ x_n=b$  with the following properties:
\begin{equation}\label{eq2}
 f^{-1}_{\beta_{-i}}(x_{i-1} )\mathcal{E}x_{i}, \ \textnormal{for} \ \textnormal{all} \ i\in\{1,\dots,n\}; \ \textnormal{and} \ f^{-1}_{\beta_{-n-1}}(x_n)\tilde{\mathcal{E}}q.
\end{equation}
\begin{theorem B}\label{thmB}
Let $\mathcal{F}=\{f_1,\dots,f_k\}\subset\mathrm{Hom}(X)$ be a finite family of homeomorphisms on $X$.
Suppose that there are two finite words
$\alpha_1\dots\alpha_s$ and $\beta_{-t}\dots\beta_{-1}$ so that for every point $x\in X$ the following
property holds:
$$
x\mathcal{E}^i_{\alpha_1\dots\alpha_s}\tilde{\mathcal{E}}q \ \mathrm{~or~}\ x\mathcal{E}^{-j}_{\beta_{-t}\dots\beta_{-1}}\tilde{\mathcal{E}}^{-1}q^\prime;\ \mathrm{for~some}\ 1\leq i \leq s,\ 1\leq j \leq t\ \mathrm{and~for~some}\ q,q^\prime\in \mathcal{Q}.
 $$
Then every $x\in X$ has the property $\mathcal{Q}$ in direction of $(\omega_0\omega_1\omega_2\dots)$ or
$(\omega_{-1}\omega_{-2}\dots)$, for every $\omega=(\dots\omega_{-1}\omega_0\omega_1\dots)\in \Sigma_k$ which has a dense orbit under the shift map $\sigma$.
\end{theorem B}
\begin{definition}
We say that $x$ and $y$ are proximal if there exists $\omega \in \Sigma_k^+$ such that
$$\liminf_{n\to\infty} d(f^n_\omega(x),f^n_\omega(y))=0.$$ Then $(x,y) \in X\times X$ is called a proximal pair.
The iterated function system $\textnormal{IFS}(X;f_1,\ldots,f_k)$ is called proximal if each pair $(x,y) \in X\times X$ is proximal. Furthermore, the $\textnormal{IFS}(X;f_1,\ldots,f_k)$ is proximal in a strong way if for every  $x,y\in X$, there exists $\Omega_0=\Omega_0(x,y) \subset
\Sigma_k^+$ with $\mathbb{P}^+(\Omega_0)=1$  such that
\begin{equation}
\label{eq:1} \liminf_{i\to\infty}d(f^i_\omega(x),f^i_\omega(y))=0,
\quad \text{for all} \ \omega\in \Omega_0.
\end{equation}
\end{definition}
The following theorem is a consequence of Theorem A.
\begin{theorem C}
\label{thmC} Let $\mathcal{F} \subset \mathrm{Hom}(S^n)$ be a
finite family of homeomorphisms of the $n$-sphere $S^n$ and let $\textnormal{IFS}(S^n;\mathcal{F})$ be the iterated function system generated by $\mathcal{F}$. Suppose
that the following assumptions hold:
\begin{enumerate}
\item \label{it:2220} $\textnormal{IFS}(S^n;\mathcal{F})$ is backward minimal, and
\item  \label{it:1110} $\textnormal{IFS}(S^n;\mathcal{F})$ contains a homeomorphism $\phi$ with exactly two fixed
points, one attracting $p$ and one repelling $q$.
\end{enumerate}
Then $\textnormal{IFS}(S^n;\mathcal{F})$ is proximal in a strong way.
\end{theorem C}
\textbf{This work is organized as follows:}
In Section 2, first of all,  we will give some applications of Theorem A and then we will prove Theorem C in the last of this section. The proofs of Theorem A and Theorem B will take Section 3.
\section{Some applications of the main result and proof of Theorem C}\label{aaaa}
$\bullet$\textbf{ \textit{Density of almost all fiber-wise orbits}}.
Suppose that $ \mathcal{F}=\{f_1,\dots,f_k\}\subset\mathrm{Hom}(X)$ is a finite family of
homeomorphisms on a compact connected  metric space $X$ and  $\textnormal{IFS}(X;\mathcal{F})$ is the iterated function system generated by $\mathcal{F}$ so that it acts minimally on $X$.
We say that $\textnormal{IFS}(X;\mathcal{F})$ satisfies the \emph{probabilistic chaos game property} \cite{bg} if for each $x \in X$ there is $\Omega(X) \subset \Sigma_k^+$ with $\mathbb{P}^+(\Omega(X))=1$ so that for each
 $\omega=(\omega_1\omega_2\dots\omega_n\ldots)\in\Omega(X)$
the chaos game orbit $\mathcal{O}^+(x,\omega)$ corresponding to the branch
$\omega$ is dense in $X$.

Now, we focus our study to the probabilistic chaos game property of minimal IFSs which obtains by Theorem A.

Indeed, suppose $\{B_i\}$  is a countable basis of $X$. For every $i\in \mathbb{N}$, one can define  two relations $\mathcal{E}(i)=\mathcal{E},\tilde{\mathcal{E}}(i)$ as follows:
\begin{enumerate}
\item\label{it:1010} $x\mathcal{E}y\ \ \Longleftrightarrow\ \ y=f_i(x)$ for some $1\leq i\leq k$;
\item\label{it:2020} $x\tilde{\mathcal{E}}(i) y\ \ \Longleftrightarrow\ \ x,y\in B_i$.
\end{enumerate}
Take $\mathcal{Q}_i:=B_i$, for $i\in \mathbb{N}$. Since $\textnormal{IFS}(X;\mathcal{F})$ is minimal,
every point $x$ has a chain connection to $\mathcal{Q}_i$, for $i\in \mathbb{N}$, with respect to the relation $\mathcal{E}$
in direction of some $\omega$.
Moreover, the continuity of generators and the openness of  $\mathcal{Q}_i$, for every $i\in \mathbb{N}$, imply that this chain connection
is stable at every point $x$.

Thus, Theorem A implies that for every $x\in X$ there exists
$\Omega_i(x)\subset\Sigma^+_k$
 with $\mathbb{P}^+(\Omega_i(x)) = 1$ such that $\mathcal{O}^+(x,\omega)\bigcap B_i\neq\emptyset$, for all $\omega\in\Omega_i(x)$ and for every $i\in \mathbb{N}$.
Take $\Omega(x):=\bigcap_i\Omega_i(x)$.
\begin{corollary}
Suppose that $ \mathcal{F}=\{f_1,\dots,f_k\}\subset\mathrm{Hom}(X)$ is a finite family of
homeomorphisms on a compact connected  metric space $X$ for which the associated iterated function system $\textnormal{IFS}(X;\mathcal{F})$ is minimal.
Then $\textnormal{IFS}(X;\mathcal{F})$ satisfies the probabilistic chaos game property; that is for every $x$, there exists a subset $\Omega(x)$ of $\Sigma^+_k$ with $\mathbb{P}^+(\Omega(x)) = 1$
so that $\overline{\mathcal{O}^+(x,\omega)}=X$, for all $\omega\in\Omega(x)$.
\end{corollary}

$\bullet$\textbf{ \textit{Density of almost all fiber-wise Chain iterates}}.
Here, we provide another application of the main result of this paper for chain transitive iterated function systems which is useful in shadowing lemmas.

Let  $\textnormal{IFS}(X;\mathcal{F})$ be an iterated function system generated  by a finite set $\mathcal{F}=\{f_1,\dots,f_k\}$ of homeomorphisms defined on a compact connected metric space $X$.

Fix $\delta> 0$. A point $x$ is
 $\delta$-\emph{chain iterate to} $y$ \emph{with respect to} $\omega$ (and we write $x\dashv_\delta^\omega y$ or $x\dashv_\delta y$) if there is a sequence of points
 $x=x_0,x_1,\dots, x_n=y$ with $n\geq1$ such that  $d(x_{i+1}, f^i_\omega(x_i ))<\delta$,
 for each
$i \in\{0, 1, 2,\dots, n -1\}$. Also, $x$ and $y$  are $\delta$-chain equivalent ($x\vdash\dashv_\delta y$) if and only if $x\dashv_\delta y$ and $y\dashv_\delta x$.
Moreover, $x$ is a \emph{chain iterate to} $y$ (and we write $x\dashv y$) if for every $\delta>0$, $x\dashv_\delta y$. We say that $x$ and $y$  are chain equivalent ($x\vdash\dashv y$) if and only if $x\dashv y$ and $y\dashv x$.
A point $x \in X$ is \emph{chain-recurrent} for the $\textnormal{IFS}(X;\mathcal{F})$
if $x$ is chain equivalent to itself. The
set of all chain recurrent points for $\mathcal{F}$
 is denoted by $R= R\mathcal{F}$.

We say that an invariant compact set $K$ is \emph{chain transitive} if for each two points $p, q\in K$ and any
$\delta> 0$, there is a $\delta$-chain iterate contained in $K$ that joints $p$ to $q$.

Now, suppose that $X$ is chain transitive for $\textnormal{IFS}(X;\mathcal{F})$.
Thus, the relation $\dashv_\delta$ is an equivalence relation, for every $\delta>0$.
For all $\delta>0$, define
\begin{enumerate}
\item\label{it:1212} $x\mathcal{E}(\delta) y\ \ \Longleftrightarrow\ \ d(f_i(x),y)<\delta$, for some $1\leq i\leq k$;
\item\label{it:12120} $\tilde{\mathcal{E}}(\delta)=\mathcal{E}(\delta).$
\end{enumerate}

Take an arbitrary point $z$ in $X$.
Since $\textnormal{IFS}(X;\mathcal{F})$ is chain transitive, so for each $x\in X$, $\{x\}$
has a chain connection to   $\{z\}$, with respect to the relation $\mathcal{E}(\delta)$  in direction of some branch $\omega$.
Thus, there is a chain $\{x_i\}_{i=0}^{t}$ in direction of $\omega$
with respect to the relation $\mathcal{E}(\delta)$ which we will denote by
$x\mathcal{E}(\delta)^{t}_\omega z$, where $x_0=x$ and $x_t=z$,  and
it is equivalent to $x\mathcal{E}(\delta)^{t-1}_\omega\tilde{\mathcal{E}}(\delta) z$.
Moreover, the continuity of generators
implies the existence of a neighborhood $B(x,r)$ so that $f_{\omega_1}(B(x,r))\subset B(x_1,\delta)$.
Hence, $(x_0=u,x_1,\dots,x_t)$ is a chain with respect to the relation
$\mathcal{E}(\delta)$  in direction of $\omega$, for every $u\in B(x,r)$.
This implies that $x$ has a stable chain connection to $\{z\}$.

Now we apply Theorem A to conclude that for each $x\in X$ there exists
$\Omega(x,z,\delta)\subset\Sigma^+_k$
with $\mathbb{P}^+(\Omega(x,z,\delta)) = 1$ such that
$x\mathcal{E}(\delta)^{t-1}_\omega\tilde{\mathcal{E}}(\delta) z$
(or $x\dashv_{\delta}^\omega z$), for every
$\omega\in\Omega(x,z,\delta)$ and for some $t\in \mathbb{N}$.\\
Clearly, $x\mathcal{E}(2\delta)^{t-1}_\omega\tilde{\mathcal{E}}(2\delta) q$
(or $x\dashv_{2\delta} q$), for every  $q\in \mathcal{Q}_z=B(z,\delta)$,
$\omega\in\Omega(x,z,\delta)$ and for some $t\in \mathbb{N}$.

On the other hand, by compactness of $X$, there exist $z_1,\dots,z_\ell\in X$ so that
$X=\bigcup_{i=1}^\ell \mathcal{Q}_{z_i}$.
Take $\Omega(x,\delta):=\bigcap_i\Omega(x,z_i,\delta)$. Clearly, $\mathbb{P}^+(\Omega(x,\delta)) = 1$.
So, for every $y\in X$ and  $\omega\in\Omega(x,\delta)$ there exists $t\in \mathbb{N}$ so that $x\mathcal{E}(2\delta)^t_\omega\tilde{\mathcal{E}}(2\delta) y$.

Finally we take $\Omega(x):=\bigcap_{i\in \mathbb{N}}\Omega(x,1/i)$.
Hence, $\mathbb{P}^+(\Omega(x)) = 1$. Therefore, for every $y\in X$
$\omega\in\Omega(x)$ and $\delta>0$, there exist $i,t\in \mathbb{N}$ such that $2/i<\delta$ and $x\mathcal{E}(2/i)^t_\omega\tilde{\mathcal{E}}(2/i) y$.
So we obtain the following corollary.
\begin{corollary}
Suppose that $ \mathcal{F}=\{f_1,\dots,f_k\}\subset\mathrm{Hom}(X)$ is a finite family of
homeomorphisms on a compact connected  metric space $X$ so that $X$ is chain transitive for $\textnormal{IFS}(X;\mathcal{F})$.
Then for every $x$, there exists a subset $\Omega(x)$ of $\Sigma^+_k$ with $\mathbb{P}^+(\Omega(x)) = 1$
such that
  $x\dashv_{2/i}^\omega y$ for every $y\in X$, $\omega\in\Omega(x)$ and $i\in \mathbb{N}$.
\end{corollary}
Now, by using of Theorem A, we will prove Theorem C.
\begin{proof}[Proof of Theorem C]
In order to apply Theorem A, we consider the open balls $\mathcal{Q}_i=B(p,\frac{1}{i})$,
for every $i\in \mathbb{N}$. Also, for every $i\in \mathbb{N}$, we define  relations $\mathcal{E}(i)=\mathcal{E},\tilde{\mathcal{E}}(i)$ as follows:
\begin{enumerate}
\item\label{it:1010} $(x,y)\mathcal{E}(x^\prime,y^\prime)\ \ \Longleftrightarrow\ \ y=f_i(x)$ and $y^\prime=f_i(x^\prime)$ for some $1\leq i\leq k$;
\item\label{it:2020} $(x,y)\tilde{\mathcal{E}}(i) (x^\prime,y^\prime)
\ \ \Longleftrightarrow\ (x,y)\mathcal{E}(x^\prime,y^\prime)\ \text{and}\ x^\prime,y^\prime\in B_i$.
\end{enumerate}
Since $\textnormal{IFS}(X;\mathcal{F})$ is backward minimal, there is $h\in \langle f_1^{-1},\dots, f_k^{-1}\rangle^+$
so that the orbit piece $q_i=h^i(q),$ for $i=0,1,2$, consisting of three different points.
Take an open ball $W_0$ containing $q$ so that $W_i=h^i(W_0)$, for $i=0,1,2$,  are mutually disjoint.
Thus, for every two arbitrary point $x,y\in X$, one of the following possibilities holds:
\begin{enumerate}
\item\label{it:2222222} $\{x,y\}\bigcap(\bigcup_{i=0}^2 W_i)=\emptyset;$
\item\label{it:3333333} $\{x,y\}\bigcap(\bigcup_{i=0}^2 W_i)\neq\emptyset.$
\end{enumerate}
In both of the cases, there exists $j=j(x,y)\in \{0,1,2\}$ so that $\{x,y\}\subset S^n\setminus W_j$.
This implies that $\{h^{-j}(x),h^{-j}(y)\}\bigcap W_0=\emptyset$.
On the other hand, for large enough $\ell$, $\phi^\ell(S^n\setminus W_0 )\subset \mathcal{Q}_i$ and so,
$\phi^\ell\circ h^{-j}(x),\phi^\ell\circ h^{-j}(y)\in \mathcal{Q}_i$.
Moreover, the continuity of generators
ensures that  $(x,y)$ has a stable chain connection to $\mathcal{Q}_i$.

Thus, Theorem A implies that for every $(x,y)\in X$ there exists
$\Omega(x,y,i)\subset\Sigma^+_k$
with $\mathbb{P}^+(\Omega(x,y,i)) = 1$ such that
$(x,y)\mathcal{E}^t_\omega\tilde{\mathcal{E}}(i) (f^{t+1}_\omega(x),f^{t+1}_\omega(y))$, for every
$\omega\in\Omega(x,y,i)$ and for some $t\in \mathbb{N}$.

Finally we take $\Omega(x,y):=\bigcap_{i\in \mathbb{N}}\Omega(x,y,i)$.
Hence, $\mathbb{P}^+(\Omega(x,y)) = 1$. Therefore, for each $(x,y)\in X\times X$,
 $\omega\in\Omega(x,y)$ and $\delta>0$, there exist $i,t\in \mathbb{N}$ so that $1/i<\delta$ and
$$
              (x,y)\mathcal{E}(1/i)^t_\omega\tilde{\mathcal{E}}(1/i) (f^{t+1}_\omega(x),f^{t+1}_\omega(y)).
$$
So, for every $(x,y)\in X\times X$, there is a subset $\Omega(x,y)$ of $\Sigma^+_k$ for which the following holds:
for every $\delta>0$, there exists $t\in \mathbb{N}$ so that
$$
 d(f^{t}_\omega(x),f^{t}_\omega(y))<\delta;\ \ \forall\ \omega\in\Omega(x,y).
$$
This completes the proof of the theorem.
\end{proof}
\section{Proof of Theorem A}

To establish Theorem A, we need to prove two following lemmas.
\begin{lemma}\label{claim1}
Let $ \mathcal{F}=\{f_1,\dots,f_k\}\subset\mathrm{Hom}(X)$ be a finite family
of homeomorphisms on $X$ and $\mathcal{Q}$ be a subset of $X$. If for $x\in X$
there is $0<\ell=\ell(x)$ so that
$$
p:=\inf_{y\in \mathcal{O}^+(x)}\nu^\ell(\{\omega_1\dots\omega_\ell\in\Sigma_k^\ell:\
y\mathcal{E}^{i}_{\omega_1\dots\omega_\ell}\tilde{\mathcal{E}}q\ \textnormal{for} \ \textnormal{some} \
1\leq i \leq \ell \ \textnormal{and} \ q\in \mathcal{Q}\})\neq 0,
$$
then for every $x\in X$, there exists
$\Omega(x)\subset\Sigma^+_k$ with $\mathbb{P}^+(\Omega(x))= 1$
such that $\{x\}$ has a chain connection to $\mathcal{Q}$ in direction of $\omega$, for every $\omega\in\Omega(x)$.
\end{lemma}
\begin{proof}
Take
$$
    \Omega(x):=\{\omega\in\Sigma_k^+:\ \{x\}~\mathrm{has}~\mathrm{chain~connction~to ~} \mathcal{Q} \mathrm{~in~direction} \ \textnormal{of} \ \omega\}.
$$
We show that $\mathbb{P}^+(\Omega(x))=1$. To prove first we define
$$
   \Omega(x,n):=\{\omega\in\Sigma_k^+:\ x\mathcal{E}^{j}_{\omega_1}\tilde{\mathcal{E}}q\ \text{ for some}\
j\leq n+\ell \text{~and~} q\in \mathcal{Q}\}.
$$
Clearly
$\Omega(x)\supset\bigcup_{n\in
N}\Omega(x,n)$. Thus
\begin{align*}
 \nu^+(\Sigma_k^+\setminus\Omega(x,n)) &\leq (1-p)\nu^+(\Sigma_{k}^{+}\setminus \Omega(x,n-\ell))\\
&\leq(1-p)^{1+[\frac{n}{\ell}]}\to 0
\end{align*}
as $n\to \infty$. It follows that $\nu^+(\Omega(x))=1$.
\end{proof}
Consider a sequence of points $a=x_0,x_1,\dots, x_{j}$ which is a chain connection from $x_0=a$ to $x_j$
with respect to the relation $\mathcal{E}$ in direction of $\rho$. When
$(f_{\sigma^j(\rho)}(x_j ),q)\not\in\mathcal{E}$, we write $a\mathcal{E}^{j}_{\rho}\not\tilde{\mathcal{E}}q$.

\begin{lemma}\label{claim2}
Under the assumptions of Theorem A, every point $x$ has a syndetic chain connection  to $\mathcal{Q}$.
\end{lemma}
\begin{proof}
To get a contradiction, suppose that there exists a point $x$ so that it does not admit any syndetic chain connection to  $\mathcal{Q}$. Then, the set
$$
A=\{i:\ x\mathcal{E}^{i}_{\omega}\tilde{\mathcal{E}}q\ \text{for~some}\ q\in\mathcal{Q},~\textrm{~and ~some~} \omega\in \Sigma^+_k\},
$$
is not syndetic in $\mathbb{N}$. This implies the existence of positive integers $a_i$, $\omega_1, \ldots, \omega_{a_i} \in \{1, \ldots, k \}$ and $k_i\in \mathbb{N}$  with $k_i\rightarrow\infty$
such that
$$
 x\mathcal{E}^{a_i+j}_{\omega_1\ldots\omega_{a_i}\rho_1\ldots\rho_j}\not\tilde{\mathcal{E}}q,\ \mathrm{for~every}\ q\in\mathcal{Q}, \ (\rho_1\ldots\rho_j)\in\Sigma^j_k \ \textnormal{and} \ j=1,\ldots,k_i,
$$
where $\Sigma^j_k $ consists of all finite words of the length $j$ and of the alphabets $\{1, \ldots, k \}$.
By passing to a subsequence, let $f^{a_i}_{\omega_1\ldots\omega_{a_i}}(x)\rightarrow y$. Then, $y\in\overline{\mathcal{O}^+(x)}$. Fix $j\in\mathbb{N}$. Obviously, for sufficiently large $i$,
$$
 f^{a_i+j}_{\omega_1\ldots\omega_{a_i}\rho_1\ldots\rho_j}(x)\rightarrow f^{j}_{\rho_1\ldots\rho_j}(y).
$$
Thus $y\mathcal{E}^{j}_{\rho_1\ldots\rho_j}\not\tilde{\mathcal{E}}q,\ \mathrm{for~every}\ q\in\mathcal{Q}$, $j\in\mathbb{N}$ and $(\rho_1\ldots\rho_j)\in\Sigma^j_k$. This implies that  $y$ does not admit any chain connection to $\mathcal{Q}$  which is a contradiction.
\end{proof}
Now, we are ready to prove Theorem A.
\begin{proof}[Proof of Theorem A]
Under the assumptions of Theorem A, Lemma \ref{claim2} ensures that every point $x$ in $X$ has a  syndetic chain
connection  to $\mathcal{Q}$. Therefore, for $x\in X$,  there exists $\ell=\ell(x)$ so that
\begin{align*}
&\inf_{y\in \mathcal{O}^+(x)}\nu^\ell(\{\omega_1\dots\omega_\ell\in\Sigma_k^\ell:\ y\mathcal{E}^i_{\omega_1\dots\omega_\ell}\tilde{\mathcal{E}}q\ \text{for some}\ q\in\mathcal{Q} \ \text{and some}\ 1\leq i \leq \ell \})\\
&\geq (\inf_{1\leq i\leq k}\nu(i))^\ell.
\end{align*}
Now, we apply Lemma \ref{claim1} to complete the proof.
\end{proof}
The following corollary is a stronger result than Lemma \ref{claim1}.
\begin{corollary}\label{proA} Let $\mathcal{F}=\{f_1,\dots,f_k\}\subset\mathrm{Hom}(X)$ be a finite family of homeomorphisms on $X$
and $\mathcal{Q}$ be a subset of $X$. Assume that there is a word $\rho\in \Sigma^+_k$ and $\ell>0$ so that for every point $x\in X$ the following property holds:
$$
 x\mathcal{E}^i_\rho\tilde{\mathcal{E}}q;\ \mathrm{for~some}\ 1\leq i \leq \ell\ \mathrm{and}\ q\in \mathcal{Q}.
$$
Then  there exists $\Omega\subset\Sigma^+_k$ with $\mathbb{P}^+(\Omega) = 1$ such that for every $x\in X$,
$x$ has the property $\mathcal{Q}$ in direction of $\omega$, for every $\omega\in\Omega$.
\end{corollary}
\begin{proof}Take
$$
    \Omega:=\{\omega\in\Sigma_k^+:\ \text{for every}\ x\in X,~x~\mathrm{has}~\mathrm{~a~chain~connection~to}~ \mathcal{Q}\mathrm{~in~direction~of}~\omega\}.
$$
We show that $\mathbb{P}^+(\Omega)=1$. First let us take
$$
   \Omega(n):=\{\omega:\ \text{for every}\ x\in X,\ x\mathcal{E}^{j}_{\omega}\tilde{\mathcal{E}}q\ \text{ for some}\
j\leq n+\ell \text{~and~}\ q\in \mathcal{Q}\}.
$$
Clearly
$\Omega\supset\bigcup_{n\in
N}\Omega(n)$. Consider a cylinder in $\Sigma^+_k$
 around the finite word $\rho_1\dots\rho_\ell$  denoted by $C_\rho$ with $\mathbb{P}^+(C_\rho)=p$. Since
for every $x\in X$,  $x\mathcal{E}^{j}_{\omega}\tilde{\mathcal{E}}q$,  for every $\omega\in
 C_\rho$ and for some $1\leq j\leq \ell$, we get
\begin{align*}\label{g}
 \mathbb{P}^+(\Sigma_k^+\setminus\Omega(x,n)) &\leq (1-p)\mathbb{P}^+(\Sigma_{k}^{+}\setminus \Omega(x,n-\ell))\\
&\leq(1-p)^{1+[\frac{n}{\ell}]}\to 0
\end{align*}
as $n\to \infty$. It follows that $\mathbb{P}^+(\Omega)=1$.
\end{proof}
\begin{proof}[Proof of Theorem B]
Write $\omega_{-t}\dots\omega_s=\beta_{-t}\dots\beta_{-1}\alpha_0\dots\alpha_s $ and take the cylinder set
$$C_{\omega_{-t}\dots\omega_s}:=\{\theta \in\Sigma_{k}: \theta_{i}=\omega_{i},-t\leq i
\leq s \}.$$

Fix $x\in X$. Now the density of the orbit $\mathcal{O}(\sigma,\rho)$ under the shift map $\sigma$ ensures that there exists  $n\in\mathbb{Z}$
so that $\sigma^{n}(\rho)\in C_{\omega_{-t}\dots\omega_s}$. Thus, there is  $ -t\leq i <0$ or $0\leq j\leq s$ so that
$$
x\mathcal{E}^{-i}_{\sigma^{n}(\omega)_{-t}\dots\sigma^{n}(\omega)_{-1}}\tilde{\mathcal{E}}^{-1}q\ \ \ \text{or}\ \ \
x\mathcal{E}^{j}_{\sigma^{n}(\omega)_0\dots\sigma^{n}(\omega)_s}\tilde{\mathcal{E}}q^\prime,\ \ \text{for some}\ q,q^\prime\in\mathcal{Q},
$$
which completes the proof.
\end{proof}
\section*{Acknowledgments}
We are grateful to Pablo G. Barrientos, Ale Jan Homburg and Fahimeh Khosh-ahang for
useful discussions and suggestions.
\def\cprime{$'$}

\end{document}